\documentclass[12pt,reqno]{amsart}
\usepackage{amsmath,amssymb,enumerate}
\usepackage[dvipdfmx]{graphicx}
\usepackage{color}
\usepackage{comment}
\usepackage[all]{xy}
\usepackage{bm}

\newtheorem{tm}{Theorem}[section]
\newtheorem{lm}[tm]{Lemma}

\newtheorem{df}[tm]{Definition}
\newtheorem{pr}[tm]{Proposition}

\makeatletter
\newcommand{\subscripts}[3]{%
  \@mathmeasure\z@\displaystyle{#2}%
  \global\setbox\@ne\vbox to\ht\z@{}\dp\@ne\dp\z@
  \setbox\tw@\box\@ne
  \@mathmeasure4\displaystyle{\copy\tw@_{#1}}%
  \@mathmeasure6\displaystyle{{#2}_{#3}}%
  \dimen@-\wd6 \advance\dimen@\wd4 \advance\dimen@\wd\z@
  \hbox to\dimen@{}\mathop{\kern-\dimen@\box4\box6}%
}
\makeatother
\newcommand{\nn}{\nonumber}

\newcommand{\x}{\bm{\mathrm{x}}}

\newcommand{\ve}{\varepsilon}
\newcommand{\dis}{\displaystyle}

\newcommand{\dd}{\mathrm{d}}

\newcommand{\N}{\mathbb{N}}

\newcommand{\E}{\mathbb{E}}

\makeatletter
\@namedef{subjclassname@2020}{\textup{2020} Mathematics Subject Classification}
\makeatother

\allowdisplaybreaks[3]

\makeatletter
 
 \@addtoreset{equation}{section}
\makeatother

\begin{document}
%%%%%%%%%%%%%%%%%%%%%%%%%%%%%%%%%%%%%%%%%%%
%\setlength{\baselineskip}{15.5pt}
%%%%%%%%%%%%%%%%%%%%%%%%%%%%%%%%%%%%%%%%%%%
\title[A Positive linear operator via the Moran model]{
On some properties of a positive linear operator 
via the Moran model in population genetics
}
%%%%%%%%%%%%%%%%%%%%%%%%%%%%
\author[T. Aoyama]{Takahiro Aoyama}
\author[R. Namba]{Ryuya Namba}
\date{\today}
\address[T. Aoyama]{Department of Applied Mathematics, 
Faculty of Science, 
Okayama University of Science, 
1-1 Ridaicho, Kita-ku, Okayama 700-0005, Japan}
\email{{\tt{aoyama@ous.ac.jp}}}
\address[R. Namba]{Department of Mathematics,
Faculty of Science,
Kyoto Sangyo University, Motoyama, Kamigamo, Kita-ku, Kyoto, 
603-8555, Japan}
%\dedicatory{Dedicated to Professor Kenji Handa on the occasion of his 60th birthday.}
\email{{\tt{rnamba@cc.kyoto-su.ac.jp}}}
%%%%%%%%%%%%%%%%%%%%%%%%%%%%

%
%
%%%%%%%%%%%%%%%%%%%%%%%%%%%%%%%%%%%%%%%%%%%%%%%%%%%%%%
\begin{abstract}

We introduce a positive linear operator 
acting on the Banach space of all continuous functions on the unit interval 
via the Moran model studied in population genetics.
We show that this operator, named the Moran operator, uniformly approximates every continuous function on the unit interval.  
Furthermore, some limit theorems for the iterates of 
the Moran operator are obtained. 
\end{abstract}
%%%%%%%%%%%%%%%%%%%%%%%%%%%%%%%%%%%%%%%%%%%%%%%%%%%%%%

\maketitle

%\tableofcontents

%%%%%%%%%%%%%%%%%%%%%%%%%%%%%%%%%%%%%%%%%%%%%%%%%%%%%%%%%%%%%%%%%%%%%%%%%
\section{{\bf Introduction}}
\label{Sect:Introduction}
%%%%%%%%%%%%%%%%%%%%%%%%%%%%%%%%%%%%%%%%%%%%%%%%%%%%%%%%%%%%%%%%%%%%%%%%%

    Let $C([0, 1])$ be the Banach space of 
    all continuous functions on $[0, 1]$ with usual supremum norm $\|\cdot\|_\infty$. 
    One of well known positive linear operators acting on 
    $C([0, 1])$ is the {\it Bernstein operator}, 
    which is defined as follows.
    
    \begin{df}[Bernstein operator]
    \label{Def:operators-1D}
       For $n \in \N$, the Bernstein operator  
       $B_n$ is defined by 
       \[
       B_n f(x) 
        = \sum^{n}_{k=0}\binom{n}{k}x^k (1-x)^{n-k} f\left(\frac{k}{n}\right),
        \qquad f \in C([0,1]), \, x \in [0,1].
       \]
    \end{df}
    
    The Bernstein operator originally introduced in \cite{Bernstein}
    to provide a constructive proof of the celebrated 
    {\it Weierstrass' approximation theorem}. 
    
    \begin{pr}[cf.~\cite{Bernstein}]
    \label{Prop:approximation}
    For any $f \in C([0, 1])$, we have
    \[
    \lim_{n \to \infty}\|B_nf-f\|_\infty=0. 
    \]
    \end{pr}
    \noindent
    The key to show this proposition is to make use of the 
    weak law of large numbers for the binomially distributed random variables. 
    We refer to \cite[Example 5.15]{Klenke} 
    for the probabilistic proof.

    We now focus on the limit of the $k$-fold iterates $B_n^k$ of the Bernstein operator $B_n$. 
    As far as we know, 
    Kelisky and Rivlin first gave a limit theorem 
    for $B_n^k$ as $k \to \infty$.  

    \begin{pr}[cf. {\cite[Theorem 1]{KR67}}]
    \label{Prop:Kelisky-Rivlin}
        Let $n \in \N$ be fixed. 
        For $f \in C([0, 1])$, we have 
        \[
        \lim_{k \to \infty}
        \max_{x \in [0, 1]}
        |B_n^k f(x) - \{f(0)(1-x)+f(1)x\}|=0. 
        \]
    \end{pr}
    
    So far, several limit theorems 
    for $B_n^k f$ as $k=k(n)$ and $n$
    tend to infinity have been investigated extensively. 
    Karlin and Ziegler studied
    some limit theorems of the $k(n)$-times iterates of positive linear 
    operators including the Bernstein operator in \cite{KZ}.
    Particularly, Konstantopoulos, Yuan and Zazanis
    showed the following. 

    \begin{pr}[cf.~{\cite[Theorem 3]{KYZ18}}]
    \label{Prop:KYZ}
    For $f \in C([0, 1])$ and $t \ge 0$, we have 
    \begin{equation}\label{Eq:ROC}
    \lim_{n \to \infty}
    \|B_n^{\lfloor nt \rfloor}f - 
    \mathbb{E}[f(\mathsf{X}_t(\cdot))]\|_\infty=0,
    \end{equation}
    where $(\mathsf{X}_t(x))_{t \ge 0}$ is the strong solution to
    the stochastic differential equation
    \begin{equation}\label{Eq:Wright-Fisher}
            \dd {\sf X}_t(x)=\sqrt{{\sf X}_t(x)(1-{\sf X}_t(x))} \, \dd W_t, 
            \qquad {\sf X}_0(x)=x \in [0, 1].
        \end{equation}
        Here, $ (W_t)_{t \ge 0} $ is a one-dimensional 
        standard Brownian motion. 
    \end{pr}

    We note that the rate of convergence of \eqref{Eq:ROC} 
    is shown to be $O(n^{-1/2})$
    when $f$ is supposed to satisfy some additional assumptions.
    See \cite[Theorem 3.2]{HN}. 
    The diffusion process $ \big({\sf X}_t(x)\big)_{t \ge 0} $ is called 
    the {\it Wright--Fisher diffusion}, 
    which appears in a fundamental probabilistic model for the evolution of the 
    allele frequency in population genetics. 
    The infinitesimal generator of the diffusion semigroup
    $T_t f(x):=\mathbb{E}[f(\mathsf{X}_t(x))]$, $t \ge 0$, is given by 
    the second order differential operator 
    \begin{equation}
    \label{Eq:generator-WF}
    \mathcal{L}f(x)=\frac{1}{2}x(1-x)f''(x),
    \qquad f \in C^2([0, 1]), \, x \in [0, 1]. 
    \end{equation}
    See \cite{EK} for more details on diffusions arising in population genetics with extensive references therein. 

    On the other hand, some authors have tried to investigate 
    some limit theorems for various kinds of positive linear 
    operators acting on some Banach spaces consisting of 
    continuous functions. 
    One of such examples is the {\it Sz\'asz--Mirakyan operator}, 
    which is defined by 
    \[
    P_n f(x):=\sum_{k=0}^\infty \mathrm{e}^{-nx}\frac{(nx)^k}{k!}f\left(\frac{k}{n}\right), \qquad f \in C_\infty([0, \infty)), \, x \in [0, \infty), 
    \]
    where we denote by $C_\infty([0, \infty))$ the 
    Banach space of all continuous functions vanishing at infinity. 
    Akahori, Namba and Semba investigated 
    the limits of the iterates $P_n^{\lfloor nt \rfloor}$ as $n \to \infty$ and they showed in \cite{ANS} that 
    it converges to the diffusion semigroup which relates to a sort of 
    branching processes by noting that $P_nf$ can be represented as 
    the expectation of the Poisson distribution. 
    This result can be read as an analogue of Proposition \ref{Prop:KYZ} in half-line cases. 
    Yet another example of such limit theorems has been discussed in \cite{HN}. The authors introduced a multidimensional generalization
    of the Bernstein operator associated with some transition probabilities of mutation among alleles. 
    As a result, they captured the multidimensional Wright--Fisher 
    diffusion process with mutation as the limit of 
    the iterates of the multidimensional Bernstein operator. 
    We note that an infinite-dimensional generalizations were also 
    studied in \cite{HN}. They considered the limit of the multidimensional Bernstein operator as both the parameter 
    and the dimension tend to infinity and captured the measure-valued 
    {\it Fleming--Viot process} under some natural assumptions.

    These limit theorems allow us to know that
    the limit of the iterates of some classes of positive linear operators should lead to some diffusion processes 
    arising in applied mathematics such as population genetics. 
    On the contrary, in the present paper,
    we start with some known mathematical models in population genetics and study several properties of 
    a positive linear operator defined via the genetic model. 
    To our best knowledge, there seem to be no papers discussing
    positive linear operators in this point of view. 
    We are to focus on the {\it Moran model}, 
    which is also a known probabilistic models 
    in population genetics as well as the Wright--Fisher model. 
    In Section \ref{Sect:Moran operator}, we give a quick review 
    of the neutral Moran model and 
    introduce the {\it Moran operator} $\mathcal{M}_n$ via the Moran model in Definition \ref{Def:Moran operator}. 
    We also show in Theorem \ref{Thm:Moran-approximation} 
    that this Moran operator 
    uniformly approximates every  $f \in C([0, 1])$,
    which is similar to Proposition \ref{Prop:approximation}.
    In Section~\ref{Sect:Kelisky-Rivlin}, we discuss
    the Kelisky--Rivlin type limit theorems for iterates of the Moran operator itself. 
    Moreover, in Section \ref{Sect:limit theorems}, 
    we consider another kind of limit theorems 
    for iterates of $\mathcal{M}_n$ with a scaling 
    different from that of Proposition \ref{Prop:KYZ}.
    Then, we establish in Theorem \ref{Thm:Moran-semigroup-convergence}
    that the iterates converges to the 
    Wright--Fisher diffusion semigroup 
    $\mathbb{E}[f(\mathsf{X}_t(x))]$ as $n \to \infty$. 
    Its rate of convergence is also obtained by using 
    a method discussed in \cite{Namba} under some additional 
    assumptions. 
    The speed rate of \eqref{Eq:rate of convergence 2} is given by 
    $O(n^{-1})$, which turns out to be faster than that of the usual Bernstein case $O(n^{-1/2})$. 
    Furthermore, we obtain  a functional limit theorem for $\mathcal{M}_n^{\left\lfloor n(n-1)t/4\right\rfloor}$, 
    that is, the weak convergence of the Markov chain induced by the Moran operator to the Wright--Fisher diffusion as a stochastic process in Theorem \ref{Thm:Donsker}. This exactly gives 
    a much stronger convergence than that of Theorem \ref{Thm:Moran-semigroup-convergence}.  
    The conclusion together with some possible future directions is provided in Section \ref{Sect:conclusion}.

%%%%%%%%%%%%%%%%%%%%%%%%%%%%%%%%%%%%%%%%%%%%%%%%%%%%%%%%%%%%%%%%%%%%%%%%%
\section{{\bf The Moran model and the Moran operator}}
\label{Sect:Moran operator}
%%%%%%%%%%%%%%%%%%%%%%%%%%%%%%%%%%%%%%%%%%%%%%%%%%%%%%%%%%%%%%%%%%%%%%%%%

%%%%%%%%%%%%%%%%%%%%%%%%%%%%%%%%%%%%%%%%%%%%%%%%%%%%%%%%%%%%%%%%%%%%%%%%%
\subsection{The Moran model}
%%%%%%%%%%%%%%%%%%%%%%%%%%%%%%%%%%%%%%%%%%%%%%%%%%%%%%%%%%%%%%%%%%%%%%%%%

The Moran model is well known as one of fundamental 
probabilistic models in population genetics, 
which was first proposed in \cite{Moran}. 
The model is characterized as a discrete-time stochastic process 
representing the dynamics in a finite population of constant size. 
There are two kinds of allele types, say $a$ and $A$, in the population
and these two compete for survival in each generation. 

\begin{df}[Moran model]
Let $n \in \N$ denote the size of a population and 
$X_k^{(n)}$ be an integer-valued random variable which represents 
the number of individuals of allele $a$ in generation $k \in \N \cup\{0\}$. 
We say that the discrete-time stochastic process 
$\{X_k^{(n)}\}_{k=0}^\infty$ is the {\it Moran model} if 
a pair of individuals is sampled uniformly at random from the population, 
one dies and the other has exactly two offspring.  
\end{df}

The stochastic process $\{X_k^{(n)}\}_{k=0}^\infty$ is 
a time-homogeneous Markov chain with values in $\mathbb{I}:=\{0, 1, 2, \dots, n\}$ whose one-step transition probability $p^{(n)}(i, j)$, $i, j \in \mathbb{I}$, 
is given by 
\begin{align}\label{Eq:Moran-transition-probability}
    p^{(n)}(i, j) &= \mathbb{P}(X_{k+1}^{(n)}=j \mid X_k^{(n)}=i) \nn\\
    &=\begin{cases}
        \dfrac{2i(n-i)}{n(n-1)} & \text{if $j=i-1$ and $i=1, 2, \dots, n$}
        \vspace{2mm}\\
        \dfrac{2i(n-i)}{n(n-1)} & \text{if $j=i+1$ and $i=0, 1,  \dots, n-1$}
        \vspace{2mm}\\
        1-\dfrac{4i(n-i)}{n(n-1)} & \text{if $j=i$ and $i=0, 1, 2, \dots, n$}\\
        0 & \text{otherwise}.
    \end{cases}
\end{align}
For more properties of the Moran model, we refer to e.g., \cite{Etheridge}. 

%%%%%%%%%%%%%%%%%%%%%%%%%%%%%%%%%%%%%%%%%%%%%%%%%%%%%%%%%%%%%%%%%%%%%%%%%
\subsection{The Moran operator and its approximating property}
%%%%%%%%%%%%%%%%%%%%%%%%%%%%%%%%%%%%%%%%%%%%%%%%%%%%%%%%%%%%%%%%%%%%%%%%%

As seen in Section \ref{Sect:Introduction}, 
the Bernstein operator is defined in terms of the 
binomial distribution. 
We now introduce a new linear operator called the {\it Moran operator} 
via the one-step transition probability \eqref{Eq:Moran-transition-probability} of the Moran model. 

\begin{df}[Moran operator]
\label{Def:Moran operator}
    Let $n \ge 2$ be an integer. 
    Then, the Moran operator is a positive linear operator 
    acting on $C([0, 1])$ defined by 
    \[
    \begin{aligned}
        \mathcal{M}_nf(x) 
        &:=
        \dis\frac{2n^2}{n(n-1)}x(1-x)
        \left\{f\left(x-\frac{1}{n}\right)+f\left(x+\frac{1}{n}\right)\right\} \nn \\ 
        &\hspace{1cm}+\left(1-\frac{4n^2}{n(n-1)}x(1-x)\right)f(x)
        % \mathcal{M}_n f(x)
        % &:= \frac{2\lfloor nx \rfloor(n-\lfloor nx \rfloor)}{n(n-1)}
        % \left\{f\left(\frac{\lfloor nx \rfloor-1}{n}\right)+
        % f\left(\frac{\lfloor nx \rfloor+1}{n}\right)\right\} \nn \\
        % &\hspace{1cm}
        % +\left(1-\frac{4\lfloor nx \rfloor(n-\lfloor nx \rfloor)}{n(n-1)}\right)
        % f\left(\frac{\lfloor nx \rfloor}{n}\right)
    \end{aligned} 
    \]
    for $x \in [1/n, (n-1)/n]$, and 
    \[
    \begin{aligned}
    \mathcal{M}_nf(x)&:=\mathcal{M}_nf\left(\frac{1}{n}\right), \qquad x \in \left[0, \frac{1}{n}\right), \\
    \mathcal{M}_nf(x)&:=\mathcal{M}_nf\left(\frac{n-1}{n}\right), \qquad x \in \left(\frac{n-1}{n}, 1\right].
    \end{aligned}
    \]
\end{df}

We then wonder if the Moran operator also uniformly approximates 
every continuous function on $[0, 1]$ or not. 
Our first main result of the present paper 
is that the Moran operator indeed has 
such an approximating property. 

\begin{tm}
\label{Thm:Moran-approximation}
    For every $f \in C([0, 1])$, we have 
    \begin{equation}\label{Eq:unif-approx-Moran}
    \lim_{n \to \infty} \|\mathcal{M}_n f - f\|_\infty=0. 
    \end{equation}
\end{tm}

\begin{proof}
    Let $e_0(x) \equiv 1$, $e_1(x)=x$ and $e_2(x)=x^2$
    for $x \in [0, 1]$. Clearly, we have $\mathcal{M}_ne_0(x)=1$. 
    Moreover, it holds that 
    \begin{align}
        \mathcal{M}_ne_1(x)
        &=\frac{2n^2}{n(n-1)}x(1-x)
        \left\{\left(x-\frac{1}{n}\right)+
        \left(x+\frac{1}{n}\right)\right\} \nn \\
        &\hspace{1cm}
        +\left\{1-\frac{4n^2}{n(n-1)}x(1-x)\right\} x 
        =x, \qquad x \in \left[\frac{1}{n}, \frac{n-1}{n}\right]. \nn
    \end{align}
    and 
    \[
    \mathcal{M}_ne_1(x)=\frac{1}{n}, \quad x \in \left[0, \frac{1}{n}\right), \qquad
    \mathcal{M}_ne_1(x)=\frac{n-1}{n}, \quad x \in \left(\frac{n-1}{n}, 1\right].
    \]
    On the other hand, we have 
    \begin{align}
        \mathcal{M}_ne_2(x)
        &=\frac{2n^2}{n(n-1)}x(1-x) 
        \left\{\left(x-\frac{1}{n}\right)^2+
        \left(x+\frac{1}{n}\right)^2\right\} \nn \\
        &\hspace{1cm}
        +\left\{1-\frac{4n^2}{n(n-1)}x(1-x)\right\} x^2 \nn \\ 
        &=x^2+\frac{4}{n(n-1)}x(1-x), \qquad x \in \left[\frac{1}{n}, \frac{n-1}{n}\right]. \nn
    \end{align}
    and 
    \[
    \mathcal{M}_ne_2(x)=\frac{1}{n^2}+\frac{4}{n^3}, \quad x \in \left[0, \frac{1}{n}\right), \qquad
    \mathcal{M}_ne_2(x)=\left(\frac{n-1}{n}\right)^2+\frac{4}{n^3}, \quad x \in \left(\frac{n-1}{n}, 1\right].
    \]
    Since it holds that 
    \[
    \begin{aligned}
    |\mathcal{M}_ne_1(x)-e_1(x)|
    &=\begin{cases}
        0 & \text{if }x \in \dis \left[\frac{1}{n}, \frac{n-1}{n}\right] \vspace{2mm}\\
        \dis\left|x-\frac{1}{n}\right| & \text{if }x \in \dis\left[0, \frac{1}{n}\right) \vspace{2mm}\\
        \dis\left|x-\frac{n-1}{n}\right| & \text{if }x \in \dis\left(\frac{n-1}{n}, 1\right],
    \end{cases} \qquad x \in [0, 1],
    \end{aligned}
    \]
    and
    \[
    |\mathcal{M}_ne_2(x)-e_2(x)|
    =\begin{cases}
        \dfrac{4}{n(n-1)}x(1-x) & \text{if }x \in \dis \left[\frac{1}{n}, \frac{n-1}{n}\right] \vspace{2mm}\\
        \dis \left|x^2-\left(\frac{1}{n^2}+\frac{4}{n^3}\right)\right| & \text{if }x \in \dis\left[0, \frac{1}{n}\right) \vspace{2mm}\\
        \dis \left|x^2-\left\{\left(\frac{n-1}{n}\right)^2+\frac{4}{n^3}\right\}\right| & \text{if }x \in \dis\left(\frac{n-1}{n}, 1\right],
    \end{cases} \qquad x \in [0, 1],
    \]
    we obtain that $\mathcal{M}_ne_0(x)=e_0(x)$ for $x \in [0, 1]$ and 
    \[
    \|\mathcal{M}_ne_i-e_i\|_\infty \le \frac{C}{n}, \qquad 
    i=1, 2, \,\, x \in [0, 1],
    \]
    for some $C>0$, which implies that 
    $\mathcal{M}_ne_i$ uniformly converges to $e_i$
    as $n \to \infty$ for $i=0, 1, 2$. 
    This allows us to apply 
    Korovkin's first theorem (cf.~\cite[Theorem 3.1]{Altomare}) 
    to deduce the uniform convergence \eqref{Eq:unif-approx-Moran}. 
\end{proof}

%%%%%%%%%%%%%%%%%%%%%%%%%%%%%%%%%%%%%%%%%%%%%%%%%%%%%%%%%%%%%%%%%%%%%%%%%
\section{{\bf The Kelisky--Rivlin type theorem for the Moran operator}}
\label{Sect:Kelisky-Rivlin}
%%%%%%%%%%%%%%%%%%%%%%%%%%%%%%%%%%%%%%%%%%%%%%%%%%%%%%%%%%%%%%%%%%%%%%%%%

Similarly to Proposition \ref{Prop:Kelisky-Rivlin}, 
we can establish the Kelisky--Rivlin type limit theorem 
for the Moran operator as well, which is the second main result 
of the present paper. 
We note that the proof is based on a fully probabilistic argument 
demonstrated in \cite[Theorem~2]{KYZ18}. 

\begin{tm}
\label{Thm:Kelisky-Rivlin-Moran}
For every $f \in C([0, 1])$, we have 
\[
\lim_{k \to \infty}
\max_{x \in [0, 1]}
|\mathcal{M}_n^k f(x) - \{f(1)x+f(0)(1-x)\}|=0. 
\]
\end{tm}

\begin{proof}
For $x \in [1/n, (n-1)/n]$, 
let $T^{(n)}(x)$ be the random variable whose law is given by 
\[
\begin{aligned}
  \mathbb{P}\left(T^{(n)}(x)=x-\frac{1}{n}\right)&=
  \frac{2n^2}{n(n-1)}x(1-x), \\
  \mathbb{P}\left(T^{(n)}(x)=x+\frac{1}{n}\right)&=
  \frac{2n^2}{n(n-1)}x(1-x), \\
  \mathbb{P}\left(T^{(n)}(x)=x\right)&=
  1-\frac{4n^2}{n(n-1)}x(1-x). 
\end{aligned}
\]
Here, we regard the law of $T_n(x)$, $x \in [0, 1/n)$, as that of $T_n(1/n)$,
and the law of $T_n(x)$, $x \in ((n-1)/n, 1]$, as that of $T_n((n-1)/n)$, 
respectively. 
Then, we have 
\[
\mathcal{M}_n f(x)=\mathbb{E}\left[f\left(T^{(n)}(x)\right)\right], 
\qquad f \in C([0, 1]), \, x \in [0, 1]. 
\]
Let $T^{(n)}_1, T^{(n)}_2, \dots$, be the independent copies of 
$T^{(n)}$ and we define a random function 
$\mathcal{Z}_k^{(n)} \colon [0, 1] \to [0, 1]$, $k=1, 2, 3, \dots$, by 
$\mathcal{Z}_k^{(n)}:=T_k^{(n)} \circ \cdots \circ T^{(n)}_2 \circ T_1^{(n)}$. 
Then, one has 
\[
\mathcal{M}_n^k f(x)=\mathbb{E}[f(\mathcal{Z}^{(n)}_k(x))],
\qquad f \in C([0, 1]), \, x \in [0, 1]. 
\]
By the independence of $\{T^{(n)}_k\}_{k=1}^\infty$, 
we also see that the sequence $\{\mathcal{Z}_k^{(n)}(x)\}_{k=0}^\infty$
can be regarded as a Markov chain 
with values in $[0, 1]$ whose one-step transition probability is given by
\[
\begin{aligned}
    p^{(n)}(\xi, \eta) &:= \mathbb{P}(\mathcal{Z}_{k+1}^{(n)}(x)=\xi \mid \mathcal{Z}_k^{(n)}(x)=\eta) \nn\\
    &=\begin{cases}
        \dfrac{2n^2}{n(n-1)}\eta(1-\eta) & \text{if $\xi=\eta - \dfrac{1}{n}$ and $\xi \ge 0$}
        \vspace{2mm}\\
        \dfrac{2n^2}{n(n-1)}\eta(1-\eta) & \text{if $\xi=\eta + \dfrac{1}{n}$ and $\xi \le 1$}
        \vspace{2mm}\\
        1-\dfrac{4n^2}{n(n-1)}\eta(1-\eta) & \text{if $\xi=\eta$ and $\xi \in [0, 1]$}\\
        0 & \text{otherwise}.
    \end{cases}
    \end{aligned}
    \]
Here, we put $\mathcal{Z}_0^{(n)}(x):=x$. 
Moreover, we can verify that 
both $0$ and $1$ are absorbing states for 
$\{\mathcal{Z}_k^{(n)}(x)\}_{k=0}^\infty$. 
We now define the first hitting time $\tau^{(n)}(x)$ by
\[
\tau^{(n)}(x):=\inf \{ k \in \N \mid 
\mathcal{Z}_k^{(n)}(x) \in \{0, 1\}\}, \qquad x \in [0, 1].
\]
Then, by the general theory of Markov chains, 
the equality
$\mathbb{P}(\tau^{(n)}(x)<+\infty)=1$ holds. 
Therefore, we obtain 
\[
\mathcal{M}_n^k f(x)=\mathbb{E}[f(\mathcal{Z}^{(n)}_k(x))]
=\mathbb{E}[f(W^{(n)}(x))], \qquad f \in C([0, 1]), \, x \in [0, 1],
\]
for all but finitely many $k \in \N$, where 
$W^{(n)}(x)$ is a random variable with values in $\{0, 1\}$. 
This leads to 
\[
\lim_{k \to \infty}\mathcal{M}_n^k f(x)
=\lim_{k \to \infty}\mathbb{E}[f(\mathcal{Z}^{(n)}_k(x))]
=\mathbb{E}[f(W^{(n)}(x))]
, \qquad f \in C([0, 1]), 
\]
uniformly in $x \in [0, 1]$. 
By taking in particular $f(x)=x$ and by noting 
$\mathbb{E}[\mathcal{Z}^{(n)}_k(x)]=x$, 
one sees that
\[
x
=1 \times \mathbb{P}(W^{(n)}(x)=1)+0 \times \mathbb{P}(W^{(n)}(x)=0),
\]
which readily implies that 
\[
\mathbb{P}(W^{(n)}(x)=1)=1-\mathbb{P}(W^{(n)}(x)=0)=x.
\]
Hence, we obtain
\[
\lim_{k \to \infty}\mathcal{M}_n^k f(x)
=f(1)x+f(0) (1-x),
\]
uniformly in $x \in [0, 1]$. 
\end{proof}

Theorem \ref{Thm:Kelisky-Rivlin-Moran} 
seems to be interesting in that the linear function
$f(1)x+f(0)(1-x)$ appears as the limit of the iterates 
of the Moran operator, similarly to the Kelisky--Rivlin theorem 
for the Bernstein operator.

%%%%%%%%%%%%%%%%%%%%%%%%%%%%%%%%%%%%%%%%%%%%%%%%%%%%%%%%%%%%%%%%%%%%%%%%%
\section{{\bf Limit theorems for iterates of the Moran operator 
and the Wright--Fisher diffusion}}
\label{Sect:limit theorems}
%%%%%%%%%%%%%%%%%%%%%%%%%%%%%%%%%%%%%%%%%%%%%%%%%%%%%%%%%%%%%%%%%%%%%%%%%

% Moreover, if $f$ is Lipschitz continuous on $[0, 1]$, 
% we then put 
%     \[
%     \mathrm{Lip}(f):=\sup_{\substack{x \neq y, x, y \in [0, 1]}}
%     \frac{|f(x)-f(y)|}{|x-y|} \, (<\infty).  
%     \]

\subsection{The iterates of the Moran operator and the Wright--Fisher diffusion}
For $k \in \N$, Let $C^k([0, 1])$ be the set of all $k$-times differentiable 
functions such that the $k$-th derivative $f^{(k)}$ is continuous on $[0, 1]$. 
In this section, we establish a limit theorem for the iterates of the
Moran operator $\mathcal{M}_n$ when the number of iterates is related with $n$. As seen below, the scaling of the number of iterates is different from that of Proposition \ref{Prop:KYZ}, 
while we capture the Wright--Fisher diffusion given by \eqref{Eq:Wright-Fisher} as 
the limit. 
The third main result of the present paper is stated as follows. 

\begin{tm}
\label{Thm:Moran-semigroup-convergence}
    For every $f \in C([0, 1])$ and $t \ge 0$, we have 
    \begin{equation}
    \label{Eq:convergence-semigroup}
        \lim_{n \to \infty}
        \left\| \mathcal{M}_n^{\left\lfloor \frac{n(n-1)}{4}t \right\rfloor}f
        -\mathbb{E}[f(\mathsf{X}_t(\cdot))]\right\|_\infty=0. 
    \end{equation}
    Moreover, if we assume that $f \in C^3([0, 1])$ and 
    $\mathrm{e}^{t\mathcal{L}}f \in C^3([0, 1])$
    for each $t \ge 0$, we then have 
    \begin{align}
    &\left\| \mathcal{M}_n^{\left\lfloor \frac{n(n-1)}{4}t \right\rfloor}f
    -\mathbb{E}[f(\mathsf{X}_t(\cdot))]\right\|_\infty \nn \\
    &\le \left(\sqrt{\frac{t}{n(n-1)}}+\frac{4}{n(n-1)}\right)
    \left(\|\mathcal{L}f\|_\infty+\frac{\|f'''\|_\infty}{12n}\right)
    +\frac{1}{12n}\int_0^t \|(\mathrm{e}^{s\mathcal{L}}f)'''\|_{\infty} \, \dd s \label{Eq:rate of convergence 2}
    \end{align}
    for $n \ge 2$ and $t \ge 0$. 
\end{tm}

In order to show Theorem \ref{Thm:Moran-semigroup-convergence}, 
we need to prove the following 
{\it Voronovskaya-type theorem} for the Moran operator
together with its rate of convergence. 
This claims the uniform convergence of
the infinitesimal generator of the discrete semigroup 
$\mathbb{E}[f(\mathcal{Z}_k^{(n)}(x))]$ to 
the one given by \eqref{Eq:generator-WF}
under a suitable scaling. 

\begin{lm}
    For $f \in C^2([0, 1])$, we have 
    \begin{equation}
    \label{Eq:generator-convergence}
        \lim_{n \to \infty}\left\|\frac{n(n-1)}{4}(\mathcal{M}_n-I)f-\mathcal{L}f\right\|_\infty=0. 
    \end{equation}
    Moreover, if $f \in C^3([0, 1])$,
    we have 
    \begin{equation}
    \label{Eq:rate of convergence}
    \left\|\frac{n(n-1)}{4}(\mathcal{M}_n-I)f-\mathcal{L}f\right\|_\infty
    \le \frac{\|f'''\|_\infty}{12n}, \qquad x \in [0, 1], \, n \ge 2. 
    \end{equation}
\end{lm}
    
\begin{proof}
    By definition, we find
    \[
    \mathbb{E}\left[T^{(n)}(x)\right]=x, 
    \qquad \mathbb{E}\left[\left(T^{(n)}(x)-x\right)^2\right]
    =\frac{4x(1-x)}{n(n-1)}.
    \]
    Hence, we have 
    \begin{align}
        &\frac{n(n-1)}{4}(\mathcal{M}_n-I)f(x) \nn \\
        &=\frac{n(n-1)}{4}\mathbb{E}\left[f\left(T^{(n)}(x)\right)-f(x)\right] \nn \\
        &=\frac{n(n-1)}{4}\mathbb{E}\left[
        \int_0^1 (1-t) (T^{(n)}(x)-x)^2
        f''\left(x+t(T^{(n)}(x)-x)\right) \, \dd t 
        \right] \nn
    \end{align}
    by using the Taylor formula with integral remainder.
    This leads to
    \begin{align}
        \frac{n(n-1)}{4}(\mathcal{M}_n-I)f(x)-\mathcal{L}f(x) %\nn \\
        &=\frac{n(n-1)}{4}
        \mathbb{E}\Bigg[
        \int_0^1 (1-t) (T^{(n)}(x)-x)^2 \nn \\
        &\hspace{1cm}\times
        \left\{f''\left(x+t(T^{(n)}(x)-x)\right)
        -f''(x)\right\}\, \dd t 
        \Bigg] \nn \\
        &=:\frac{n(n-1)}{4}\mathbb{E}[\mathcal{I}^{(n)}(x)], 
        \qquad x \in [0, 1]. 
        \label{Eq:Taylor-1}
    \end{align}
    Since $f''$ is uniformly continuous on $[0, 1]$, 
    for any $\ve>0$, 
    there exists a sufficiently small $\delta=\delta(\ve)>0$ such that 
    $|x-y|<\delta$ implies $|f''(x)-f''(y)|<\ve$. 
    Then, we have 
    \begin{align}
        &\frac{n(n-1)}{4}|\mathbb{E}[\mathcal{I}^{(n)}(x) \, : \, |T^{(n)}(x)-x|<\delta]|\nn  \\
        &\le \frac{\ve n(n-1)}{8} \mathbb{E}\left[\left(T^{(n)}(x)-x\right)^2\right] \nn \\
        &=\frac{\ve n(n-1)}{8} \times \frac{4x(1-x)}{n(n-1)}=\frac{\ve}{8},
        \label{Eq:Taylor-2}
    \end{align}
    where we used $x(1-x) \le 1/4$ for all $x \in [0, 1]$. 
    On the other hand, we see that 
    \begin{align}
        &\frac{n(n-1)}{4}|\mathbb{E}[\mathcal{I}^{(n)}(x) \, : \, |T^{(n)}(x)-x| \ge \delta]| \nn \\
        &\le \frac{\|f''\|_\infty }{2}n(n-1) \mathbb{E}\left[
        \int_0^1 (1-t) (T^{(n)}(x)-x)^2 \, \dd t \, : \, 
        |T^{(n)}(x)-x| \ge \delta
        \right]\nn \\
        &\le \frac{\|f''\|_\infty }{4}\left(1-\frac{1}{n}\right) 
        \mathbb{P}(|T^{(n)}(x)-x| \ge \delta) \nn \\
        &=\frac{\|f''\|_\infty }{4}\left(1-\frac{1}{n}\right)  
        \times \frac{1}{\delta^2}\frac{4x(1-x)}{n(n-1)} \nn \\
        &\le \frac{\|f''\|_\infty}{4\delta^2n(n-1)}
        \left(1-\frac{1}{n}\right),
        \label{Eq:Taylor-3}
    \end{align}
    by applying the Chebyshev inequality. 
    Then, it follows from \eqref{Eq:Taylor-1}, \eqref{Eq:Taylor-2} and \eqref{Eq:Taylor-3} that 
    \[
    \left|\frac{n(n-1)}{4}(\mathcal{M}_n-I)f(x)-\mathcal{L}f(x)\right|
    \le \frac{\ve}{8}+\frac{\|f''\|_\infty}{4\delta^2n(n-1)}
        \left(1-\frac{1}{n}\right),
    \]
    for $x \in [0, 1]$ and $n \ge 2$, which concludes \eqref{Eq:generator-convergence} after letting $n \to \infty$ and $\ve \searrow 0$. 

    We next suppose that $f \in C^3([0, 1])$. 
    Then, Taylor's formula gives us 
    \[
    \begin{aligned}
    &\left|\frac{n(n-1)}{4}(\mathcal{M}_n-I)f(x)-\mathcal{L}f(x)\right| \\
    &\le \frac{n(n-1)}{4}\left|\mathbb{E}\left[
        \int_0^1 (1-t)^2 (T^{(n)}(x)-x)^3
        f'''\left(x+t(T^{(n)}(x)-x)\right) \, \dd t 
        \right]\right|  \\
    &\le \frac{n(n-1)}{12}\|f'''\|_\infty
    \mathbb{E}\Big[ |T^{(n)}(x)-x|^3\Big]
    \le \frac{\|f'''\|_\infty}{12n}
    \end{aligned}
    \]
    for $x \in [0, 1]$ and $n \ge 2$, where we used  
    \[
    \mathbb{E}\Big[|T^{(n)}(x)-x|^3\Big]=\frac{4x(1-x)}{n^2(n-1)} \le \frac{1}{n^2(n-1)}
    \]
    in the third line. 
\end{proof}

We are now ready for the proof of Theorem \ref{Thm:Moran-semigroup-convergence}. 
Note that the latter part of the proof is based on 
\cite[Theorem~1]{Namba}. 

\begin{proof}[Proof of Theorem {\rm \ref{Thm:Moran-semigroup-convergence}}]
    It is known that the closure $\overline{\mathcal{L}}$ of the differential operator
    $\mathcal{L}$ generates the Feller semigroup 
    on $C([0, 1])$ and $C^2([0, 1])$ is a core for the closure
    (see e.g., \cite{AC}). 
    In particular, the Lumer--Phillips theorem implies that 
    the closure of $\mathcal{L}$ is dissipative. 
    Hence, we see that $I-\overline{\mathcal{L}}$ is invertible. 
    Since $C^2([0, 1])$ is a core for $\overline{\mathcal{L}}$, 
    we know that $(I-\overline{\mathcal{L}})(C^2([0, 1]))$
    is dense in $C([0, 1])$. Moreover, it is not difficult to show 
    that $\overline{\mathcal{L}}$ is densely defined. 
    Therefore, the celebrated Trotter theorem (cf. \cite{Trotter, Kurtz})
    leads to 
    \[
    \lim_{n \to \infty}
        \left\| \mathcal{M}_n^{\left\lfloor \frac{n(n-1)}{4}t \right\rfloor}f
        -\mathrm{e}^{t \mathcal{L}}f\right\|_\infty=0
    \]
    for every $f \in C([0, 1])$ and $t \ge 0$.
    We recall that the infinitesimal generator of the Wright--Fisher 
    diffusion semigroup $T_tf(x)=\mathbb{E}[f(\mathsf{X}_t(x))]$ 
    is also given by $\mathcal{L}$. Thus, the uniqueness of the 
    Feller semigroup implies the uniform convergence \eqref{Eq:convergence-semigroup}. 

    We next suppose that $f \in C^3([0, 1])$ and $\mathrm{e}^{t\mathcal{L}}f \in C^3([0, 1])$ for each $t \ge 0$. 
    Let us put 
    \[
    S_t^{(n)}:=\mathrm{e}^{\frac{n(n-1)}{4}t(\mathcal{M}_n-I)}, \qquad t \ge 0.
    \] 
    Then, the triangle inequality and $\mathbb{E}[f(\mathsf{X}_t(x))]=\mathrm{e}^{t\mathcal{L}}f(x)$, $f \in C([0, 1])$,  
    leads to 
    \[
    \begin{aligned}
        &\left\|\mathcal{M}_n^{\left\lfloor \frac{n(n-1)}{4}t \right\rfloor}f
        -\mathbb{E}[f(\mathsf{X}_t(\cdot))]\right\|_\infty  \\
        &= \left\|\mathcal{M}_n^{\left\lfloor \frac{n(n-1)}{4}t \right\rfloor}f
        -S_{\frac{4}{n(n-1)}\left\lfloor \frac{n(n-1)}{4}t \right\rfloor}^{(n)}f\right\|_\infty
        +\left\|S_{\frac{4}{n(n-1)}\left\lfloor \frac{n(n-1)}{4}t \right\rfloor}^{(n)}f - S_t^{(n)}f\right\|_\infty \\
        &\hspace{1cm}+\left\|S_t^{(n)}f - \mathrm{e}^{t\mathcal{L}}f\right\|_\infty 
        =:A_1^{(n)}+A_2^{(n)}+A_3^{(n)}. 
    \end{aligned}
    \]
    We give an estimate of $A_1^{(n)}$. 
    By applying \cite[Lemma III.5.1]{Pazy} and \eqref{Eq:rate of convergence}, we have 
    \begin{align}
    A_1^{(n)} &= \left\|\mathcal{M}_n^{\left\lfloor \frac{n(n-1)}{4}t \right\rfloor}f
        -\mathrm{e}^{\left\lfloor \frac{n(n-1)}{4}t \right\rfloor(\mathcal{M}_n-I)}f\right\|_\infty \nn \\
        &\le \sqrt{\left\lfloor \frac{n(n-1)}{4}t \right\rfloor}
        \|(\mathcal{M}_n-I)f\|_\infty \nn \\
        &\le \sqrt{\left\lfloor \frac{n(n-1)}{4}t \right\rfloor}\frac{4}{n(n-1)}
        \left\|\frac{n(n-1)}{4}(\mathcal{M}_n-I)f\right\|_\infty \nn \\
        &\le 2\sqrt{\frac{t}{n(n-1)}}\left(\|\mathcal{L}f\|_\infty
        +\frac{\|f'''\|_\infty}{12n}\right). 
        \label{Eq:A_1}
    \end{align}
    Moreover, the term $A_2^{(n)}$ is estimated as
    \begin{align}
    A_2^{(n)} &= \left\|\int_{t}^{\frac{4}{n(n-1)}\left\lfloor \frac{n(n-1)}{4}t \right\rfloor}
    S_s^{(n)}\left(\frac{n(n-1)}{4}(\mathcal{M}_n-I)\right)f \, \dd s\right\|_\infty \nn \\
    &\le \left|\frac{4}{n(n-1)} \left\lfloor \frac{n(n-1)}{4}\right\rfloor - t\right| \times 
    \left\|\frac{n(n-1)}{4}(\mathcal{M}_n-I)f\right\|_\infty \nn \\
    &\le \frac{4}{n(n-1)}\left(\|\mathcal{L}f\|_\infty
        +\frac{\|f'''\|_\infty}{12n}\right).
        \label{Eq:A_2}
    \end{align}
    On the other hand, \eqref{Eq:rate of convergence} gives us that 
    \begin{align}
    A_3^{(n)} &= \left\| -\int_0^t \frac{\dd}{\dd s}
    (S_{t-s}^{(n)}\mathrm{e}^{s\mathcal{L}})f \, \dd s \right\|_\infty \nn \\
    &= \left\|\int_0^t S_{t-s}^{(n)}
    \left(\frac{n(n-1)}{4}(\mathcal{M}_n-I)-\mathcal{L}\right)\mathrm{e}^{s\mathcal{L}}f \, \dd s\right\|_\infty  \nn \\
    &\le \int_0^t \left\|\left(\frac{n(n-1)}{4}(\mathcal{M}_n-I)-\mathcal{L}\right)(\mathrm{e}^{s\mathcal{L}}f)\right\|_\infty \, \dd s \nn \\
    &\le \frac{1}{12n}\int_0^t \|(\mathrm{e}^{s\mathcal{L}}f)'''\|_\infty \, \dd s. 
        \label{Eq:A_3}
    \end{align}
    Therefore, \eqref{Eq:A_1}, \eqref{Eq:A_2} and \eqref{Eq:A_3}
    allow us to obtain the desired convergence rate. 
\end{proof}

\subsection{The tightness of the Markov chain induced by the Moran operator}

In a probabilistic point of view, 
Theorem \ref{Thm:Moran-semigroup-convergence} tells us that 
the random variable 
$\mathcal{Z}_{\lfloor n(n-1)t/4 \rfloor}^{(n)}(x)$
weakly converges to $\mathsf{X}_t(x)$ as $n \to \infty$
for fixed $t \ge 0$ and $x \in [0, 1]$. 
We also establish a more stronger convergence result 
than that of Theorem \ref{Thm:Moran-semigroup-convergence} in this section. 
Let $(\mathbf{Z}_t^{(n)})_{t \ge 0}$, $n=2, 3, 4, \dots$, 
be the sequence of continuous-time
stochastic process defined via the linear interpolation
\[
\mathbf{Z}_t^{(n)}(x)
:=\mathcal{Z}_{\left\lfloor \frac{n(n-1)}{4}t \right\rfloor}^{(n)}(x)
+\left(\frac{n(n-1)}{4}t -
\left\lfloor \frac{n(n-1)}{4}t \right\rfloor\right)
\left(\mathcal{Z}_{\left\lfloor \frac{n(n-1)}{4}t \right\rfloor+1}^{(n)}(x)-\mathcal{Z}_{\left\lfloor \frac{n(n-1)}{4}t \right\rfloor}^{(n)}(x) \right),
\]
for $t \ge 0$ and $x \in [0, 1]$. 
Moreover, for $x \in [0, 1]$, we put 
\[
C^{(1/2)-}_x([0, 1])
:=\left\{ \varphi \in C([0, 1]) \, \Bigg| \, \varphi(0)=x, \, 
\sup_{\substack{t \neq s, t, s \in [0, 1]}}
\frac{|\varphi(t)-\varphi(s)|}{|t-s|^\alpha}<\infty \text{ for }\alpha<1/2\right\}.
\]
Then, our final main result is stated as follows. 

\begin{tm}
    \label{Thm:Donsker}
For every $x \in [0, 1]$, 
the sequence $\{\mathbf{Z}_\cdot^{(n)}(x)\}_{n=2}^\infty$
converges in law to the Wright--Fisher diffusion $(\mathsf{X}_t(x))_{0 \le t \le 1}$ as $n \to \infty$
in the H\"older space
$C_x^{(1/2)-}([0, 1])$.  
\end{tm}

In order to show Theorem \ref{Thm:Donsker}, 
it suffices to show the following two claims. 
One is the finite-dimensional distribution of
$\mathbf{Z}_t^{(n)}(x)$ converges to that of 
$\mathsf{X}_\cdot(x)$ for $x \in [0, 1]$.
However, this easily follows from 
Theorem \ref{Thm:Moran-semigroup-convergence}
by noting e.g., \cite[Theorem 17.25]{Kallenberg}. 
The other is to show the following. 

\begin{lm}
For $x \in [0, 1]$, the sequence $ \{\mathbb{P} \circ (\mathbf{Z}^{(n)}_{\cdot})^{-1}\}_{n=2}^{\infty} $
of image measures
is tight in $C_x^{(1/2)-}([0, 1])$. 
\end{lm}

\begin{proof}
  Our aim is to show the existence of some positive constant 
  $ C > 0 $ independent of $n$ such that
    \begin{equation}
    \label{Eq:tight}
        \E\left[
        \left|
        \mathbf{Z}^{(n)}_{t}(\x) - \mathbf{Z}^{(n)}_{s}(\x)
        \right|^{2m}\right]
        \le C(t-s)^{m}, \qquad
        n \in \N, \,\, 0 \le s \le t, \, m \in \N.
    \end{equation}
    As soon as \eqref{Eq:tight} is established, 
    the Kolmogorov continuity criterion 
    implies that the sequence $ \{\mathbb{P} \circ (\mathbf{Z}^{(n)}_{\cdot}(x))^{-1}\}_{n=2}^{\infty} $
    is tight in $C([0, 1])$ and that $\mathbf{Z}_t^{(n)}(x)$
    has an $\alpha$-H\"older continuous version
    for all $\alpha<(m-1)/2m$. 
    Since $m \in \N$ can be chosen arbitrarily, 
    we complete the proof. 
    
    \vspace{2mm}
    \noindent
    {\bf Step 1.}
    Let $m \in \N$. 
    We here show that there is some $C>0$ independent of $n$ such that
    \begin{equation}
    \label{tight-1}
        \E\left[\left|
        \mathbf{Z}^{(n)}_{\frac{4\ell}{n(n-1)}}(x) - \mathbf{Z}^{(n)}_{\frac{4k}{n(n-1)}}(x)
        \right|^{2m}\right]
        \le
        C\left(\frac{4(\ell-k)}{n(n-1)}\right)^m,
        \qquad k, \ell=0, 1, 2, \dots,\,\, k \le \ell.
    \end{equation}
    
    for some $C > 0$ independent of $n$. 
    We note that the Markov chain 
    $ \{\mathcal{Z}^{(n)}_k(x)\}_{k=0}^\infty $ is a martingale
    with respect to the natural filtration. 
    Indeed, we have 
    \[
    \begin{aligned}
    &\E\left[\mathcal{Z}^{(n)}_{k+1}(x)
    -\mathcal{Z}^{(n)}_k(x) \, \Big|\, \mathcal{Z}^{(n)}_k(x)=\eta\right] \\
    &=\frac{2n^2}{n(n-1)}\eta(1-\eta)\left\{
    \left(\eta-\frac{1}{n}\right)+\left(\eta+\frac{1}{n}\right)\right\}\\
    &\hspace{1cm}
    +\left(1-\frac{4n^2}{n(n-1)}\eta(1-\eta)\right)\eta - \eta=0,
    \qquad \eta \in [0, 1].
    \end{aligned}
    \]
    Hence, the Burkholder--Davis--Gundy inequality gives us that
    \begin{align}
        \E\left[\left|
        \mathbf{Z}^{(n)}_{\frac{4\ell}{n(n-1)}}(x) - \mathbf{Z}^{(n)}_{\frac{4k}{n(n-1)}}(x)
        \right|^{2m}\right]
        &=
        \E\left[\left|
        \mathcal{Z}^{(n)}_{\ell}(x) - \mathcal{Z}^{(n)}_{k}(x)
        \right|^{2m}\right]
        \notag\\
        &\le
        C_{\mathrm{BDG}}\,
        \E\left[ \left(\sum_{j=k}^{\ell-1}
        \left|\mathcal{Z}_{j+1}^{(n)}(x) 
        - \mathcal{Z}_j^{(n)}(x)\right|^2
        \right)^{m}\right]\label{tight-1A},
    \end{align}
    where $ C_{\mathrm{BDG}}>0$  is the constant appearing in the  Burkholder--Davis--Gundy inequality with 
    exponent $ 2m $.  
    By using the Markov property, one gets 
    \begin{align}
        &\E\left[ \left(\sum_{j=k}^{\ell-1}
        \left|\mathcal{Z}_{j+1}^{(n)}(x) 
        - \mathcal{Z}_j^{(n)}(x)\right|^2
        \right)^{m}\right] \nn \\
        &= (\ell - k)^m \, 
        \E\left[
        \left|T^{(n)}(x) - x\right|^{2m}
        \right]\notag\\
        &= (\ell-k)^m \times 
        \frac{4n^2}{n(n-1)}x(1-x)\left(\frac{1}{n}\right)^{2m} \nn \\
        &\le  \left(\frac{4(\ell-k)}{n(n-1)}\right)^m
        \times \left(\frac{n(n-1)}{4}\right)^m \times 
        \frac{4n^2}{n(n-1)} \times \frac{1}{4} \times \left(\frac{1}{n}\right)^{2m} \nn \\
        &\le \left(\frac{4(\ell-k)}{n(n-1)}\right)^m
        \times \left(\frac{n}{2}\right)^{2m} \times \frac{n}{n-1}
        \times \left(\frac{1}{n}\right)^{2m} 
        =\frac{1}{2^{2m-1}}\left(\frac{4(\ell-k)}{n(n-1)}\right)^m
        .\label{tight-1B}
    \end{align}
    By combining \eqref{tight-1A} with 
    \eqref{tight-1B}, we arrive at \eqref{tight-1}. 
    
    \vspace{2mm}
    \noindent
    {\bf Step 2.}
    We now prove \eqref{Eq:tight}.
    Let $0 \le s \le t$. 
    We take $ 0 \le k \le \ell $ with
    \[
    \frac{4k}{n(n-1)} \le s < \frac{4(k+1)}{n(n+1)}, \qquad
    \frac{4\ell}{n(n-1)} \le t < \frac{4(\ell+1)}{n(n+1)}.
    \]
    By the definition of $\mathbf{Z}_\cdot^{(n)}(x)$, 
    we have 
    \begin{align*}
        \left|\mathbf{Z}^{(n)}_{\frac{4(k+1)}{n(n+1)}}(x) - \mathbf{Z}^{(n)}_{s}(x)\right|
        &=
        \left(k+1 - \frac{n(n-1)}{4}s\right)
        \left|\mathbf{Z}^{(n)}_{\frac{4(k+1)}{n(n+1)}}(x) - \mathbf{Z}^{(n)}_{\frac{4k}{n(n+1)}}(x)\right|, \\
        \left|\mathbf{Z}^{(n)}_{t}(x) - \mathbf{Z}^{(n)}_{\frac{4\ell}{n(n+1)}}(x)\right|
        &=
        \left(\frac{n(n-1)}{4}t - \ell\right)
        \left|\mathbf{Z}^{(n)}_{\frac{4(\ell+1)}{n(n+1)}}(x) - \mathbf{Z}^{(n)}_{\frac{4\ell}{n(n+1)}}(x)\right|.
    \end{align*}
    Hence, it follows from \eqref{tight-1} and the triangle inequality that
    \begin{align*}
        &\E\left[
        \left|
        \mathbf{Z}^{(n)}_{t}(x) - \mathbf{Z}^{(n)}_{s}(x)
        \right|^{2m}\right]\\
        &\le 
        3^{2m-1}\Bigg\{
        \left(\frac{n(n-1)}{4}t - \ell\right)^{2m} \times C\left(\frac{4}{n(n-1)}\right)^{m}
        + C\left(\frac{4(l - k - 1)}{n(n-1)}\right)^{m} \\
        &\hspace{1cm}
        + \left(k+1 - \frac{n(n-1)}{4}s\right)^{2\beta} \times C\left(\frac{4}{n(n-1)}\right)^{\beta}
        \Bigg\}\\
        &\le 
        C\left\{
        \left(t - \frac{4\ell}{n(n-1)}\right)^m 
        + \left(\frac{4\ell}{n(n-1)} - \frac{4(k+1)}{n(n-1)}\right)^m + 
        \left(\frac{4(k+1)}{n(n-1)} - s\right)^m
        \right\} \\
        &\le C(t - s)^{m}
    \end{align*}
    for all $n=2, 3, 4, \dots$, where $C>0$ is a positive constant independent of $n$. 
\end{proof}

%%%%%%%%%%%%%%%%%%%%%%%%%%%%%%%%%%%%%%%%%%%%%%%%%%%%%%%%%%%%%%%%%%%%%%%%%
\section{{\bf Conclusion}}
\label{Sect:conclusion}
%%%%%%%%%%%%%%%%%%%%%%%%%%%%%%%%%%%%%%%%%%%%%%%%%%%%%%%%%%%%%%%%%%%%%%%%%

Throughout the present paper, 
we have introduced the 
Moran model by a known genetic model 
and have discussed some of its properties such as 
the approximating property and certain limit theorems 
for iterates of the Moran operator. 
We believe that our results lead to unknown objects in the study of 
functional analysis, approximation theory and probability theory
where various probabilistic models may provide 
positive linear operators which possess, 
in some sense, nice approximating properties which we may obtain some new interesting limit theorems relevant to the 
theory of diffusion processes.

Furthermore, 
we might capture some jump processes arising in the study of 
population genetics through the limits of the iterates of 
some positive linear operators. 
It is pointed out in \cite{BB} that such phenomena with jumps 
are highly relevant to genetic models beyond finite variance. 
In order to discuss these kinds of models, 
it is convenient for us to consider the 
{\it Cannings model} introduced in \cite{CanningsI, CanningsII}. 
This model describes the dynamincs of a size-fixed population 
 with non-overlapping generations. 
 Note that the Moran model belongs to the class of Cannings models. 
 Generally, these models are defined by using the notion of the 
 {\it exchangeable} random variables indicating the number of 
 offspring of each individual. 
 
 Therefore, our next goal is to find some candidates of 
 exchangeable random variables to define new linear operators
 whose iterates converges to some jump processes arising in population genetics.
 Remark that we may find some essential candidates in e.g., \cite{EW, HM13}
 from genetic perspectives which would help us to go on. 
 We aim to provide a general framework 
 to connect some classes of positive linear operators
 with some classes of jump processes 
 via scaling limits of the iterates of operators
 by focusing these existing examples.

\vspace{2mm}
\noindent
{\bf Acknowledgements}. 
The second-named author is supported by JSPS KAKENHI 
Grant No.~23K12986.

% References

%%%%%%%%%%%%%%%%%%%%%%%%%%%%%%%%%
%%%%%%%%%%%%%%%%%%%%%%%%%%%%%%%%%
%%%%%%%%%%%%%%%%%%%%%%%%%%%%%%%%%

%%%%%%%%%%%%%%%%%%%%%%%%%%%%%%%%%%%%%%%%%%%%%%%
\end{document}